\documentclass[12pt]{amsart}
\makeatletter					
\skip\footins=4ex					
\makeatother					
\usepackage{amsmath,amsthm}			
\usepackage{amssymb}
\usepackage{lastpage}
\usepackage[pdftex]{graphicx}			
\usepackage{hyperref}				
\usepackage{epstopdf}
\usepackage{color}
\usepackage{hyperref}

\newtheorem{theorem}{Theorem}[section]
\newtheorem{lemma}[theorem]{Lemma}
\newtheorem{corollary}[theorem]{Corollary}
\newtheorem{proposition}[theorem]{Proposition}

\newtheorem{remark}[theorem]{Remark}

\newtheorem{definition}[theorem]{Definition}

\title{\vspace*{-5mm} Classification of eventually periodic subshifts}

\author{\sc{Benjam\'in A.~Itz\'a-Ortiz}}
\address{Universidad Aut\'onoma del Estado de Hidalgo, Centro de Investigaci\'on en Matem\'aticas, Pachuca, Hidalgo 42180, \sc{Mexico}}
\email{itza@uaeh.edu.mx}

\author{\sc{Meghan B.~Malachi}}
\address{Providence College, Department of Mathematics and Computer Science, Providence, RI 02918, \sc{USA}}
\email{mmalachi@friars.providence.edu}

\author{\sc{Austin Marstaller}}
\address{The University of Texas at Dallas, Department of Mathematical Sciences, Richardson, TX 75080, \sc{USA}}
\email{Austin.Marstaller@utdallas.edu}

\author{\sc{Jason Saied}}
\address{ Lafayette College, Mathematics Department, Easton PA, 18042,  \sc{USA}}
\email{saiedj@lafayette.edu}

\author{\sc{Sara Stover}}
\address{Mercer University, Department of Mathematics, Macon GA 31207, \sc{USA}}
\email{sara.mae.stover@live.mercer.edu}

\begin{document}

\maketitle

\begin{abstract}
We provide a classification of eventually periodic subshifts up to conjugacy and flow equivalence. We use our results to prove that each skew Sturmian subshift is conjugate to exactly one other skew Sturmian subshift and that all skew Sturmian subshifts are flow equivalent to one another.
\end{abstract}

\section*{Introduction}

In this paper we study non-periodic bi-infinite sequences of symbols that, when some finite word is removed, yield periodic bi-infinite sequences. We term these eventually periodic sequences and define an anomaly word as the word that, when removed, leaves the sequence periodic. By considering the closure of the orbit of a given bi-infinite sequence, we obtain a subshift, which is a symbolic dynamical system. It is customary to name a subshift after the sequence from which it is created; in this fashion, we refer to periodic subshifts, eventually periodic subshifts, and so on.  While the classification of periodic subshifts is well understood and even straightforward, 
eventually periodic subshifts are  strictly sofic, so their classification is somewhat more elaborate.  We prove that the conjugacy class of an eventually periodic subshift depends only on the size of the anomaly and the least period of its corresponding periodic sequence and also that, as is the case for periodic subshifts, there is a single flow equivalence class of eventually periodic subshifts.

Motivation for this work was given by the skew Sturmian sequences introduced in \cite{hedlund}, which are a particular class of eventually periodic sequences. Sturmian sequences, especially those which are not periodic nor eventually periodic, have attracted the attention of many mathematicians \cite{berthe,clark,darnell,fogg,kurka,lothaire}. However,  skew Sturmian sequences are important to study since they appear naturally in the generalization of properties of Sturmian sequences \cite{berstel,b-v,glen,pirillo}. By showing in this paper how to compute the size of an anomaly word of a skew Sturmian sequence, we are able to apply our results to obtain a characterization for conjugacy of skew Sturmian subshifts. We show that all skew Sturmian subshifts have a conjugacy class of size two, regardless of their least period. In contrast, we also show that all skew Sturmian subshifts belong to the same flow equivalence class, since they are all eventually periodic.

We divide this work into three sections. In the first section we introduce the necessary notions and notations required in the paper. In the second  section we state the classification results concerning conjugacy and flow equivalence of eventually periodic subshifts. Finally, in the last section we apply our results to classify skew Sturmian subshifts.

This research was carried out while the authors were participating in the Research Experience for Undergraduates program at the Califonia State Unversity Channel Islands (CSUCI), supported by NSF grant DMS-1359165.

\section{Preliminaries}
In this section we introduce the necessary terminology used in the paper. For the symbolic dynamics notation and basic results we follow  \cite{lind}.
Throughout this paper $\mathcal A$ denotes a finite set, called an alphabet, and its elements are referred as its  symbols or letters. The space ${\mathcal A}\sp\mathbb{Z}$ of all bi-infinite sequences in $\mathcal A$ is a metric space referred to as the full shift. We view elements in $\mathcal A$ as sequences $x=(x_i)$ with $i\in\mathbb Z$ and $x_i\in\mathcal A$. By a word (or block) of length $n$ we mean a finite sequence of $n$ symbols of ${\mathcal A}\sp\mathbb{Z}$. It is customary to denote by $|w|$ the length of a word $w$. Given $x\in{\mathcal A}\sp\mathbb{Z}$ we say that a word $w$ occurs (or is allowed) in $x$ if  there is an $i\in\mathbb Z$ and an $n\geq 0$ such that $w=x_ix_{i+1}\cdots x_{i+n}=x_{[i,i+n]}$. In addition, we say that a word $w$ is forbidden in $x$ if it does not occur in $x$.
Naturally, a subword of a word $w$ is a word itself that occurs in $w$.

The shift map  $\sigma\colon{\mathcal A}\sp\mathbb{Z}\to{\mathcal A}\sp\mathbb{Z}$ defined by $\sigma(x)_i=x_{i+1}$ for $i\in\mathbb Z$ is a homeomorphism implementing a  natural action of $\mathbb Z$ on ${\mathcal A}\sp\mathbb{Z}$. A subshift or shift space $X$ is a subset of ${\mathcal A}\sp\mathbb{Z}$ that is both closed and $\sigma$-invariant. Equivalently,  a subset  $X$ of the full shift is a subshift if there is a set of words $\mathcal F$ such that $X$ consists of every  bi-infinite sequence for which the words in $\mathcal F$ are forbidden.

If $x, y\in{\mathcal A}\sp\mathbb{Z}$ and there exists $k\in\mathbb{Z}$ with $\sigma^k(x) = y$, then we say that $x$ and $y$ are similar sequences.

We say that $x\in{\mathcal A}\sp\mathbb{Z}$ is periodic with period $N>0$ if $\sigma^N(x)=x$. If $N$ is the least integer satisfying this property then we say that $N$ is the least period of $x$. In particular, when $N=1$ we say that $x$ is a fixed point. If $x$ is a periodic bi-infinite sequence and $w$ is an allowed word of $x$ of length $N$, where $N$ is the least period of $x$, then there is $\ell\in\mathbb Z$ such that $x_{[\ell+kN,\ell+(k+1)N-1]}=w$ for all $k\in\mathbb Z$. We shall refer to $w$ as a repeating or periodic word for $x$. Notice that the $N$ allowed different repeating words for $x$ are exactly the $N$ possible repeating words for any sequence similar to $x$.

\begin{definition}
\normalfont We say that a non-periodic bi-infinite sequence $x\in{\mathcal A}\sp\mathbb{Z}$ is eventually periodic if there exist integers $i$ and $n>0$ so that by removing the allowed word $v = x_{[i+1,i+n]}$
from $x$ the resulting bi-infinite sequence $p_v(x)$, defined as
$$p_v(x)_k=\begin{cases}
 x_k & \text{ if } k\leq i \\
 x_{k+n} & \text{ if } k>i,
\end{cases}$$ is a periodic sequence. We call $v$ an anomaly word for $x$, and will say that $|v|$ is the anomaly size of $x$ with respect to $v$.
\end{definition}

 We remark that our definition is slighly different to the ones we were able to find in the literature, for instance, \cite[p.~156]{fogg} allows the sequence preceding and following the allowed word to have different periods; however, for the case of skew Sturmian sequences, using \cite[Exercise~6.2.3, p.~157]{fogg} we see that our defintion of eventually periodic bi-infinite sequence is equivalent to the given there.

We first prove that the least period of an eventually periodic sequence $x$ is independent of the anomaly word.

\begin{lemma}\label{similarperiodic}
Let $x$ be an eventually periodic sequence and suppose that $u$ and $v$ are anomaly words for $x$. Then $p_u(x)$ is equal to $p_v(x)$.
\end{lemma}
\begin{proof}  
   Since by definition $u$ and $v$ are allowed words in $x$, there exist integers $i$ and $j$ such that $x_i$ and $x_j$ are the symbols preceding $u$ and $v$ in $x$. Let $k=\min\{i,j\}$.
Then $p_u(x)$ and $p_v(x)$ are both periodic sequences which agree for all indices less or equal to $k$ and thus they are equal.
\end{proof}

The previous lemma then allows us to give the following.

\begin{definition}
\normalfont
Let $x$ be an eventually periodic sequence. 
We define the least period of $x$ as the least period of the periodic sequence $p_v(x)$ for some anomaly word $v$ of $x$.
\end{definition}

Just as the repeating word of a periodic sequence is not unique, the anomaly word of an eventually periodic sequence is also not unique. Let $w_1$ and $w_2$ be the repeating words $123$ and $231$, respectively, and let $u$ and $v$ be the anomaly words $00$ and $23001$. Then the sequences $\cdots w_1w_1uw_1w_1\cdots$ and $\cdots w_2w_2vw_2w_2\cdots$ are similar but have different anomaly words. We will show in the next proposition, however, that the lengths of the anomaly words are always congruent modulo the least period.

\begin{proposition}\label{indepOfWord}
Let $x$ be an eventually periodic sequence and suppose that $N$ is the least period of $x$. If  $u$ and $v$ are anomaly  words for $x$ then $|u|\equiv |v| \mod N$.
\end{proposition}
\begin{proof}
Since $u$ and $v$ are anomaly words, there are integers $i$, $j$, and $m,n>0$ such that $u=x_{[i+1,i+m]}$ and $v=_{[j+1,j+n]}$. If $|u|=|v|$ there is nothing to prove. Therefore we may assume  $|u|\not =|v|$. Without loss of generality, suppose that $|v|<|u|$. Consider the periodic words $w_1=x_{[i-N+1,i]}$ and $w_2=x_{[j-N+1,j]}$ of $p_u(x)$ and $p_v(x)$, respectively. Using Lemma~\ref{similarperiodic}, there are subwords $s$ and $t $ of $w_1$ such that $w_1=st$ and $w_2=ts$. Since $x=\cdots w_1w_1uw_1w_1\cdots$, we verify that the word $v_1=tus$ is an allowed word for $x$; furthermore, $x=\cdots w_2w_2v_1w_2w_2\cdots$. Since $x=\cdots w_2w_2vw_2w_2\cdots$ as well, and because $p_v(x)=p_{v_1}(x)$ by Lemma~\ref{similarperiodic}, it must follow that for some integer $k>0$, $|u|+N=|v_1|=|v|+k|w_2|=|v|+kN$. Thus $|u|\equiv |v| \mod N$, as was to be proven.  
\end{proof} 

It will be useful to identify the smallest anomaly size. We formalize this in the following.

\begin{definition}
 \normalfont  Let $x$ be an eventually periodic sequence with least period $N$. Let $v$ be an anomaly of $x$. We define the anomaly size of $x$ as
\[
  a(x)=\min\left\{|u|
  \colon  u \text{ is an anomaly of } x \right\}.
\]
\end{definition}

Given a subshift $X$ of ${\mathcal A}\sp\mathbb{Z}$ and $n>0$ we denote by $\mathfrak B_n\left(X\right)$ the collection of all words of length $n$ which are allowed in the sequences in $X$. Fix integers $m$ and $n$ so that $-m\leq n$ and let $\Phi \colon\mathfrak B_{m+n+1}\left(X\right)\to \mathcal A$. Then $\Phi$ defines a mapping $\phi\colon X\to {\mathcal A}\sp\mathbb{Z}$, called a sliding block code (with memory $m$ and anticipation $n$), by the formula $\phi\left(x\right)_i=\Phi\left(x_{\left[i-m,i+n\right]}\right)$..

Given  $x\in{\mathcal A}\sp\mathbb{Z}$ the orbit of $x$ is the set ${\mathrm{Orb}}(x)=\left\{\sigma^k(x)\colon k\in \mathbb Z\right\}$, that is to say, the orbit of $x$ is the set consisting of all sequences similar to $x$. The closure $X$ of  ${\mathrm{Orb}}(x)$ is a subshift called the subshift generated by $x$. When $x$ is an eventually periodic bi-infinite sequence, we accordingly call its generated subshift an eventually periodic subshift. It is easy to see that the subshift generated by an eventually periodic sequence consists of the orbit of $x$ and the orbit of $p_v(x)$ for some anomaly word $v$ in $x$.

Recall that two subshifts are said to be conjugate if there is a bijective sliding block code from one to the other. This is equivalent to the definition of conjugacy for topological dynamical systems. Conjugacy is an important equivalence relation among dynamical systems; however, it is often too strong a condition.

Finally, we recall the definition of flow equivalence. Let  $X$ be  a subshift and consider the product $X\times \mathbb R$. In this product space we introduce an equivalence relation defined by $(x,s)\sim(y,t)$ when $y=\sigma^n(x)$ and $t=s-n$, for $n\in\mathbb Z$. We then define the suspension of $X$ as the quotient space $\Sigma X =X\times {\mathbb R}\slash\sim$. 
Consider the the quotient map $\pi\colon X\times \mathbb R\to \Sigma X$. It is customary to denote an equivalence class by $[x,s]$ rather than $\pi(x,s)$. A natural flow on $\Sigma X$ is defined as $t\mapsto [x,s+t]$. Two subshifts $X$ and $Y$ are said to be flow equivalent if there is a homeomorphism from  $\Sigma X$ to $\Sigma Y$ which preserves the orientation of flow lines.

\section{Classification Results}

In this section we prove the classification results. We begin with the conjugacy problem.

\begin{lemma}\label{conjToAnom}
Let $X$ (respectively $Y$) be the subshift generated by the eventually periodic sequence $x$ (respectively $y$) with least period $N$ (respectively $M$). If $X$ and $Y$ are conjugate, then $M = N$ and $a(x)\equiv a(y)\mod N$.
\end{lemma}
\begin{proof}
Assume there exists a conjugacy $\phi: X\rightarrow Y$.  Let $u$ (respectively $v$) be an anomaly word for $x$ (respectively $y$) such that $|u|=a(x)$ (respectively $|v|=a(y)$). Assume without loss of generality that $a(x)\leq a(y)$. Observe that $p_u(x)\in X$ and $p_v(y)\in Y$. Then because $\phi$ maps periodic points to periodic points and preserves their least periods \cite[Proposition~1.5.11]{lind}, we obtain that $\phi(p_u(x))$ is similar to $p_v(y)$ and  hence $M=N$. Next we will show that $x$ has an anomaly word of length $a(y)$. An application of
Proposition~\ref{indepOfWord} will then complete the proof of our lemma. Using again the fact that $\phi$ is a conjugacy and that neither $x$ nor $y$ are periodic, there exists an integer $k$ such that $\phi\left(\sigma^k(x)\right)=\sigma^k(\phi(x))=y$. Furthermore, since $v$ is an allowed word of length  $a(y)$ there is an integer $j$ such that  $v=y_{[j+1,j+a(y)]}$.  Therefore it follows that
\begin{align*}
y_{[j-N+1,j]}&=y_{[j+a(y)+1,j+a(y)+N]}
\end{align*}
as both are repeating words of $p_v(y)$  of size $N$ immediately preceding and following the anomaly word $v$ of $y$. On the other hand
\[
    y_{[j-N+1,j]}=\sigma^k(\phi(x)_{[j-N+1,j]})=\phi(x)_{[j-N+1+k,j+k]}
\]
and similarly
\[
    y_{[j+a(y)+1,j+a(y)+N]}=
    \phi(x)_{[j+a(y)+1+k, j+a(y)+N+k]}
\]
whence we obtain $x_{[j+k+1,j+a(y)+k]}$ is an anomaly word for $x$ of size $a(y)$, as was to be proven.
\end{proof}

\begin{proposition}\label{anomToConj}
Let $X$ and $Y$ be subshifts generated by eventually periodic sequences $x$ and $y$ respectively, each with least period $N$. If $a(x)\equiv a(y)\mod N$, there exists a conjugacy $\phi: X\rightarrow Y$.
\end{proposition}
\begin{proof}
   Let $u$ and $v$ be anomaly words of $x$ and $y$, respectively, such that $a(x)=|u|$ and $a(y)=|v|$. Without loss of generality, assume $|u|\leq |v|$. Then, by hypothesis, there is an integer $k$ such that $|v|=|u|+kN$. Since $u$ and $v$  are allowed words, there are integers $i$ and $j$  such that $u=x_{[i,i+a(x)-1]}$ and $v=y_{[j,j+a(y)-1]}$. Let us define a sliding block code $\phi\colon X\to Y$ with memory and anticipation $|v|$ induced by the $2|v|+1$-block map $\Phi$ given by
\[
  \Phi\left(x_{[l-|v|,l+|v|]}\right)=y_{j-i+l},\,\,\, l\in\mathbb Z.
\]
It follows that $\phi(x)$ is similar to $y$ and therefore it gives the desired conjugacy.
\end{proof}

\begin{theorem}\label{anomIffConj}
For subshifts $X$ and $Y$ generated by eventually periodic sequences $x$ and $y$ respectively, $X$ and $Y$ are conjugate if and only if $x$ and $y$ have the same least period $N$ and the same anomaly size mod $N$.
\end{theorem}
\begin{proof}
This follows from Lemma~\ref{conjToAnom} and Proposition~\ref{anomToConj}.
\end{proof}

We may also classify eventually periodic sequences up to flow equivalence.

\begin{theorem}\label{floweq}
All subshifts generated by eventually periodic sequences are flow equivalent.
\end{theorem}
\begin{proof}
Let $x$ in ${\mathcal A}\sp\mathbb Z$ be an eventually periodic sequence and assume it has least period $N$.
We first show that there are eventually periodic sequences $y_1$ and $z_1$,  both generating subshifts flow equivalent to the subshift generated by $x$, such that $y_1$ has least period $N+1$ and anomaly size $a(x)$ whilst $z_1$ has least period $N$ and anomaly size $a(x)+1$. For this purpose,
suppose that $u$ is an anomaly of $x$ such that $a(x)=|u|$ and let $w$ be the word of length $N$ immediately following $u$ in $x$.
Further assume that $a$ and $b$ are   the last letters for the words  $u$ and $w$, respectively.  Let $a\sp\prime$ and $b\sp\prime$ be symbols not in $\mathcal A$.
Consider now the words $u_1$ and $w_1$ obtained from $u$ and $w$ by replacing their last symbols $a$ and $b$ by $a\sp\prime$ and $b\sp\prime$, respectively. Consider now the eventually periodic sequence $y_a$ with anomaly word $u$ and least period $N$ such that the repeating word for $p_u(x)$ immediately following $u$ in $x_a$ is $w_1$. Similarly, we consider the eventually periodic sequence $z_b$ with anomaly word $u_1$ and repeating word following $u_1$ in $x_b$ to be $w$. If follows from Theorem~\ref{anomIffConj} that the subshift generated by $x$ is conjugate to the subshifts generated by  $x_a$ and $x_b$, respectively. Now let $\alpha$ and $\beta$ be symbols not belonging to the sets $\mathcal A\cup\{a\}$ and $\mathcal A\cup \{b\}$, respectively. By means of  symbol expansion, we consider new sequences $y_1$ and $z_1$ obtained from $y_a$ and $z_b$ by substituting each symbol $a$ and $b$ by words $a\alpha$ and $b\beta$, respectively. Since flow equivalence is generated by conjugacy and symbol expansion \cite{parry} we then get that the subshifts generated by $y_1$ and $z_1$ are flow equivalent 
to the subshift generated by $x$, as was to be proved. For a more detailed account on symbol expansions, the reader is referred to \cite[Section~1.3]{danishmaster} or \cite[Chapter~2]{danishmaster}.

Using a straightforward induction argument we conclude that, given any positive integers $i$ and $j$, there exist eventually periodic sequences $y_i$ and $z_j$ such that $y_i$ has least period $N+i$ and anomaly size $a(x)$ and $z_j$ has least period $N$ and anomaly size $a(x)+1$.

For the proof of the theorem, consider arbitrary eventually periodic subshifts $X$ and $X\sp\prime$ and suppose they have as generators the eventually periodic sequences $x$ and $x\sp\prime$, respectively. Further assume $X$ and $X\sp\prime$  have least periods $M$ and $M\sp\prime$, respectively. Let $N=\max\{M,M\sp\prime\}$. By the work above, there are sequences $y$ and $y\sp\prime$ with induced subshifts flow equivalent to $X$ and $X\sp\prime$, respectively, such that $y$ and $y\sp\prime$ have the same least period $N$ and anomaly sizes $a(x)$ and $a(x\sp\prime)$, respectively. Let $A=\max\{a(x),a(x\sp\prime)\}$. Again, by the work above, there are eventually periodic sequences $z$ and $z\sp\prime$ with induced subshifts flow equivalent to those induced by $y$ and $y\sp\prime$, respectively, such that both have least period $N$ and anomaly size $A$. An application of Theorem~\ref{anomIffConj} yields the subshifts generated by $z$ and $z\sp\prime$ to be conjugate and therefore flow equivalent. Hence $X$ and $X\sp\prime$ are flow equivalent, as wanted.
\end{proof}

\section{Skew Sturmian Sequences}

In this section we apply our results to classify skew Sturmian subshifts. We will begin by reviewing the definition of skew Sturmian sequences. We then  show that they are eventually periodic and calculate their anomaly sizes in order to apply our results.

For the remainder of this paper, let $\mathcal A$ denote the two symbol alphabet $\mathcal A=\{0,1\}$.  Sturmian sequences were introduced by Hedlund and Morse in \cite{hedlund}, where we refer the reader for the details of the theory presented next.  
Given a sequence $x\in\mathcal{A}\sp\mathbb{Z}$ we associate a cell-series which may be regarded as an abbreviation of $x$ as a bi-infinite sequence of allowed words $\cdots  B_{-1}  B_0  B_1  \cdots $ where each $B_i$, a cell, denotes an allowed word of $x$ which consists only of 0's except its first symbol which is a  symbol 1. We shall say that such a cell-series is a cell-series representation of $x$. Furthermore, we define an $n$-chain as a portion of a cell-series, namely, an allowed word of $x$ written in the form $B_{r}  \cdots B_s  $, with $r\leq s$ and $n=s-r+1$. We say that $x\in{\mathcal A}\sp\mathbb{Z}$  is Sturmian if, given arbitrary $m>0$, the difference between the numbers of $0$'s contained in any two allowed $m$-chains is at most one.

Given a cell-series representation of a Sturmian sequence $x$, it is understood that (a sequence similar to) the original $x$ might be recovered simply by replacing each word $B_i$ by the string of symbols it denotes. For this reason, we refer to a Sturmian sequence or one of its cell-series representations interchangeably. If $b_n$ denotes the total quantity of zeroes in any given $n$-chain, the limit $b_n/n$ as $n$ goes to infinity converges to a real number $\alpha$, which we define to be the frequency of $x$. We remark that the definition of our cell-series is a slight modification of the notion of cell-series defined in \cite{hedlund}, where a cell is composed entirely of $0$ symbols. We find our notation avoids the writing of the $1$ symbols between cells and does not at all affect the results.

It turns out that Sturmian sequences with irrational frequency are not periodic and that periodic Sturmian sequences have rational frequency. A skew Sturmian sequence is a non-periodic Sturmian sequence with rational frequency.

We remark that in the literature there are a number of equivalent ways to state the definition of a Sturmian sequence, e.g., see \cite{berthe} and references within. Since most authors focus on the irrational frequency case, this explains why we have decided to follow here the original definition of Morse and Hedlund.

Among the many beautiful results in \cite{hedlund} we highlight Theorems 5.1, 5.2 and 5.3. They establish algorithms to construct all possible Sturmian sequences with nonnegative real frequency $\alpha$. In particular, \cite[Theorem~5.2]{hedlund} asserts that any given Sturmian cell-series representation is produced by one of two possible algorithms. We will reproduce these next, partially because of the relevance they have to the proofs of our results, but also because a typo found in the original paper has obliged us to set the record straight.

%
%
Throughout the rest of the paper, let $\alpha=\dfrac{q}{p}$ be a positive rational number for relatively prime positive integers $p$ and $q$.
Let $\beta = \dfrac{1}{\alpha} = \dfrac{p}{q}$. Fix $m\in\mathbb{Z}$ and define the set $G = \{m+k\beta: k\in\mathbb{Z}\}$. Define then the cell-series representation $S(m,\alpha)$, where the number of zeroes in the cell $B_n$ is the number of points $x\in G$ satisfying
\begin{equation}\label{S}
n <  x \leq n+1, \,\,\,\, m<x<m+1,\,\,\,\, n\leq  x < n+1,
\end{equation}
according as $n<m$, $n=m$, or $n>m$. The cell-series representation $S\sp\prime(m,\alpha)$ is similarly defined by requiring $B_n$ to be a cell containing as many zeroes as there are points $x\in G$ satisfying
\begin{equation}\label{Sprime}
n \leq  x < n+1, \,\,\,\, m\leq x\leq m+1,\,\,\,\, n<  x \leq n+1,
\end{equation}
according as $n<m$, $n=m$, or $n>m$.

We remark that in \cite{hedlund}, the first inequality in (\ref{Sprime}) above is incorrectly stated as $n<x\leq n+1$. In fact, as stated in that paper, the resulting cell-series is not necessarily Sturmian. After carefully reviewing \cite[Theorem~4.8]{hedlund} where it is proved that both $S(m,\alpha)$ and $S\sp\prime(m,\alpha)$ are in fact skew Sturmian, we learn that the inequalities (\ref{Sprime}) were really intended.

The following lemma will be needed to compute the anomaly size of a skew Sturmian sequence. We remark that it could also be derived by means of continued fractions \cite[Theorem~2]{khinchin}.


\begin{lemma}\emph{(Restricted B\'{e}zout's Identity)}\label{bezout}
For relatively prime positive integers $p$ and $q$, there exist unique integers $0 \leq a < q$ and $0 < b \leq p$, such that  $bq - ap = 1.$ Furthermore, $a+b$ and $p+q$ are also relatively prime.
\end{lemma}
\begin{proof}
Let $p$ and $q$ be relatively prime positive integers. By B\'ezout's identity, the equation $qx + (-p)y = 1$ has a solution $(x_1, y_1)$ and all other solutions have the form
$$x = x_1 + \ell p\textnormal{ and }y = y_1 - \ell(-q) = y_1 + \ell q$$
for any $\ell\in\mathbb{Z}$ (for a proof, see \cite[Theorem 5.1]{niven}). Since all possible values of $x$ are some multiple of $p$ apart, there must be exactly one $\ell_1$ such that $0<x_1+\ell_1 p\leq p$. By the same logic, there must be exactly one $\ell_2$ such that $0\leq y_1+\ell_2 q < q$. We define $b = x_1+\ell_1 p$ and $a = y_1+\ell_2 q$. We will show these values $a$ and $b$ satisfy the equation $bq-ap = 1$. It suffices to show that $\ell_1 = \ell_2$. For the sake of contradiction, we assume that $\ell_1>\ell_2$. (The case $\ell_1<\ell_2$ is similar.) Let $a' = y_1+\ell_1 q$, and observe that $a'\geq q$. Then $1 = bq - a'p \leq bq - pq = q(b-p)<0$, which yields the desired contradiction.

Furthermore, we see that
$$(p+q)b + (a+b)(-p) = bp+bq - ap - bp = bq - ap = 1,$$
implying that $p+q$ and $a+b$ are relatively prime by \cite[Theorem 1.3]{niven}. 
\end{proof}

In accordance with the conventions introduced at the beginning of the paper, by a skew Sturmian subshift with frequency $\alpha$ we shall mean a subshift generated by a skew Sturmian sequence with frequency $\alpha$. We will further refer to a skew Sturmian sequence as being of type $S$ or $S\sp\prime$ if it has a cell-series representation $S(m,\alpha)$ or $S\sp\prime(m,\alpha)$ for $m\in\mathbb Z$, respectively.

\begin{remark}\label{stypes}
 \normalfont It is straightforward to see that for any two integers $m$ and $n$ the cell-series $S(m,\alpha)$ and $S(n,\alpha)$ (respectively $S\sp\prime(m,\alpha)$ and $S\sp\prime\left(n,\alpha\right)$) are similar. Hence when referring to a skew Sturmian sequence, up to similarity, it will suffice to specify its frequency and type only.
\end{remark}

\begin{proposition}\label{reciprocals}
Let $x$ and $y$ be skew Sturmian sequences with positive frequencies $\frac{p}{q}$ and $\frac{q}{p}$. If the cell-series representations of $x$ and $y$ are of different type then the subshifts associated to $x$ and $y$ are conjugate.
\end{proposition}
\begin{proof}
First we recall a geometric representation of non-skew Sturmian sequences by means of cutting sequences as defined in  \cite[Page 48]{lothaire}. Since it will be useful for our purposes, we will describe next how this representation can be further modified to represent skew Sturmian sequences. Consider a skew Sturmian sequence with positive real frequency $\alpha = \frac{p}{q} = \beta^{-1}$. Let $m\in\mathbb{Z}$.
Recall that we construct these sequences by superimposing an integer lattice over $\mathbb{R}^2$ and adding a $1$ or a $0$ every time the line $y = \beta x + m$ crosses a horizontal lattice line or a vertical lattice line, respectively. This may be generalized further to skew Sturmian sequences by modifying the symbols inserted at points where both horizontal and vertical lines are crossed simultaneously, which we call integer lattice points. To represent $S(m, \alpha)$, when the line crosses an integer lattice point below (or at) the $y$-value $m$, we insert $01$. When the line crosses an integer lattice point above the $y$-value $m$, we insert $10$. Similarly, to represent $S'(m, \alpha)$, we insert $10$ when the line crosses integer lattice points below (or at) the $y$-value $m$ and $01$ when the line crosses integer lattice points above the $y$-value $m$. The inequalities \ref{S} and \ref{Sprime} indicate that we then have a correspondence between the number of $0$'s in the cell $B_n$ and the $0$'s added to the sequence between the $1$'s added at $y = n$ and $y = n+1$. 

We will now show that the subshifts $X$ and $Y$ generated by the sequences $S(0,\frac{q}{p})$ and $S^\prime (0,\frac{p}{q})$, respectively, are conjugate by a simple symbol reversal.

Consider the lines associated to the two sequences, $y = \frac{p}{q}x$ and $y = \frac{q}{p}x$ respectively. Because they have reciprocal slopes, we see that every vertical lattice line crossed by the first line corresponds to a horizontal lattice line crossed by the second, and vice versa. In addition, notice that the rules for dealing with intersections with integer coordinate points are reversed. At any intersection above the origin, the first line adds $10$ while the second adds $01$; at any intersection below (or at) the origin, the first line adds $01$ while the second adds $10$.

Then we may map one sequence to the other using the $1$-block map $\Phi$ which sends $0$ to $1$ and $1$ to $0$. This yields a bijective sliding block code from $X$ to $Y$.
\end{proof}

Using this result, we see that up to conjugacy, every skew Sturmian subshifts with positive frequency may be generated by a sequence of type $S$. We may extend this result to $S'(m,0)$, the sequence of the form $\cdots 1 1 1 0 1 1 1 \cdots$, by defining $S(m,\infty)$ as the sequence $\cdots 0 0 0 1 0 0 0 \cdots $. (We note that $S(m, 0)$ and $S'(m, \infty)$ are not defined.) Therefore all skew Sturmian subshifts may be generated using a sequence of type $S$.

We will now prove a previously announced result.
 Notice that the fact that skew Sturmian sequences are eventually periodic is known see e.g. \cite[Lemma~6.2.2, p.~156]{fogg}. However, more is needed: namely, to calculate the anomaly size in order to use our classification results from the previous section.

\begin{proposition}\label{anomalysize}
Let $x$ be a skew Sturmian sequence of frequency $\alpha = \frac{q}{p}$ where $p$ and $q$ are relatively prime positive integers and suppose that $(a,b)$ are the restricted B\'ezout coefficients for $(q,p)$. If $x$ is of type $S$ or $S\sp\prime$, then it is an eventually periodic sequence with least period $p+q$ and anomaly size  $a+b$ or $p+q-\left(a+b\right)$, respectively.
\end{proposition}
\begin{proof}
We first assume that $x$ is of type $S$ with frequency $\alpha = \frac{q}{p}$. We may further assume by Remark \ref{stypes} that $x$ may be represented by $S(0,\alpha) = \cdots B_{-1} B_0 B_1\cdots$. By \cite[Theorems 3.3 and 4.8]{hedlund}, the beam preceding $B_0$ in $x$ has least cell-period $p$, and the beam following $B_0$ in $x$ has least cell-period $p$ as well. For $(a,b)$ the restricted B\'ezout coefficients of $(q,p)$, we will show that the chain $B_b\cdots B_{b+p-1}$ is identical to the chain $B_{-p}\cdots B_{-1}$, which indicates that removing $B_{0}\cdots B_{b-1}$ yields a periodic bi-infinite sequence.

First we calculate the size of this chain in terms of symbols. By \cite[Theorems 3.3 and 4.8]{hedlund}, the $p$-chain $B_0\cdots B_{p-1}$ is the unique $p$-chain containing $q-1$ $0$'s, while all others contain $q$ $0$'s. Therefore the $bp$-chain $B_0\cdots B_{bp - 1}$ contains $bq - 1 = ap$ $0$'s. Consider breaking this chain into $p$ consecutive $b$-chains $B_0\cdots B_{b-1}$, $B_b\cdots B_{2b-1}$, and so on. We see that if any of those $b$-chains did not contain precisely $a$ $0$'s, we would have a violation of the Sturmian condition. Thus the $b$-chain $B_0\cdots B_{b-1}$ consists of $a+b$ symbols.

We will now show that the chain $B_b\cdots B_{b+p-1}$ is identical to the chain $B_{-p}\cdots B_{-1}$. Recall that the number of $0$'s in the cell $B_n$ of $S(0, \alpha)$ is identical to the number of multiples of $\beta = \frac{p}{q}$ satisfying the inequalities (\ref{S}) above, with $m = 0$. Denote the set of such multiples as $\mathcal{B}_n$. Then we will show that for $r\in\{-p, \cdots -1\}$, the formula $f(z) = z+p+b-\frac{1}{q}=z+p+a\frac{p}{q}$, defines the desired bijection from $\mathcal{B}_r$ onto $\mathcal{B}_{r+p+b}$. 
 Indeed, to verify that $f(\mathcal{B}_r)\subseteq\mathcal{B}_{r+p+b}$, we let  $z = k\beta \in\mathcal{B}_r$ be arbitrary, where $k\in\mathbb{Z}$. We see that $f(z)  = (k+q+a)\dfrac{p}{q}$, so that $f$ indeed takes multiples of $\beta$ to other multiples of $\beta$. 
 Further, we know that $r < \beta k \leq r+1$. Therefore we may say that $r +p +b  < \beta k + p+ b = f(z)+\dfrac{1}{q}\leq r+p+b+1$ and so $r +p +b -\dfrac{1}{q} < f(z)\leq r+p+b+1-\dfrac{1}{q}<r+p+b+1$.
Since  $f(z)$ is a multiple of $\frac{1}{q}$, the last inequality yields $r+p+b\leq f(z)<r+p+b+1$, that is,  $f(\beta k)\in\mathcal{B}_{r+p+b}$, as wanted. The proof of the other inclusion is similar. 



We now assume that we have a skew Sturmian sequence $y$ of type $S'$ with frequency $\frac{q}{p}$. By Proposition \ref{reciprocals}, the subshift generated by $y$ is conjugate by symbol reversal to a subshift generated by a skew Sturmian sequence $z$ with frequency $\frac{p}{q}$ and type $S$. Then because this conjugacy simply exchanges $0$'s and $1$'s, we see that $y$ and $z$ have the same anomaly size. (This is slightly stronger than using the result of Lemma \ref{conjToAnom}, which would only guarantee that the sizes were the same mod $p+q$.) We must then only calculate the anomaly size of $z$. Notice that if $(a,b)$ are the restricted B\'ezout coefficients for $(q,p)$ as above, then $p(q-a) - q(p-b) = 1$.
Further, we see that $0<q-a \leq q$ and $0\leq p-b < p$. Then $(p-b, q-a)$ are the restricted B\'ezout coefficients for $(p,q)$. By the first part of the proof, since $z$ is of type $S$, we see that it has anomaly size $p-b+q-a = (p+q) - (a+b)$.
\end{proof}

Now that we may calculate the anomaly size of an arbitrary skew Sturmian sequence, we would like to use the results of Section 2 to classify them. We first need to compare skew Sturmian sequences with the same least period.

\begin{proposition}\label{noConjugacies}
Let $p$ and $q$ be relatively prime positive integers. Let $c\neq 0$ be an integer so that $q-c$ and $p+c$ are also relatively prime positive integers. Then any two skew Sturmian subshifts of the same type with frequencies $\dfrac{q}{p}$ and $\dfrac{q-c}{p+c}$ respectively have different anomaly sizes.
\end{proposition}
\begin{proof}
Let the ordered pairs $(a,b)$ and $(e,f)$ be the restricted B\'ezout coefficients of the ordered pairs $(q,p)$ and $(q-c,p+c)$, respectively.

Let $x$ and $y$ be two skew Sturmian sequences with frequencies $\dfrac{q}{p}$ and $\dfrac{q-c}{p+c}$, respectively. Let us first assume that $x$ and $y$ are both of type $S$. For the sake of contradiction, assume they have the same anomaly size. Then, by Proposition~\ref{anomalysize}, 
there must exist an integer $d$ such that  $e=a-d$ and $f=b+d$. Therefore,
\begin{align*}
1 &= (b+d)(q-c) - (a-d)(p+c)
\\&= 1 - c(a+b) + d(p+q).
\end{align*}
We obtain $d(p+q) = c(a+b)$. Notice that if $d = 0$, then it must also be that $c(a+b) = 0$, and so either $c=0$ or $a+b = 0$. Since both of these are impossible, it must be that $d\neq 0$. Since $d\neq0$, we may say that $a\neq e$ and that $b\neq f$.  Thus $\dfrac{p+q}{a+b} = \dfrac{c}{d}$, and since $p+q$ and $a+b$ are relatively prime by Lemma \ref{bezout} we conclude that $c = k(p+q)$ for some integer $k\neq 0$. This implies that either $q-c$ or $p+c$ is negative, both of which are contradictions.

By Proposition~\ref{reciprocals}, we see that the result holds for $x$ and $y$ of type $S'$ as well.
\end{proof}

Note that because $S(m, \infty)$ and $S'(m, 0)$ are the only skew Sturmian sequences of least period $1$, their subshifts cannot be conjugate to any subshifts of positive real frequency. Further, it is clear that the subshifts generated by $S(m,\infty)$ and $S'(m, 0)$ are conjugate. 

We finally prove the main results of this section. For the first corollary we consider $S(m, \infty)$ and $S'(m, 0)$ to have inverse frequencies.

\begin{corollary}
Let $X$ and $Y$ be skew Sturmian subshifts induced by non similar skew Sturmian sequences. Then $X$ and $Y$ are conjugate if and only if they have inverse frequencies and are of opposite type. 
\end{corollary}
\begin{proof}
By Proposition~\ref{reciprocals}, we see that skew Sturmian subshifts with inverse frequencies and opposite type are conjugate. 

For the other implication, let us assume that $X$ and $Y$ are conjugate skew Sturmian subshifts induced by skew Sturmian sequence $x$ and $y$, respectively. Suppose that $X$ and $Y$  have frequencies $\frac{q}{p}$ and $\frac{q\sp\prime}{p\sp\prime}$, respectively.  By Proposition~\ref{anomalysize} and Theorem~\ref{anomIffConj} then $p+q = p'+q'$ and $a(x)\equiv a(y)\mod (p+q)$. Now, combining  Propositions~\ref{noConjugacies}, \ref{anomalysize}, and \ref{reciprocals} guarantees that either $x$ and $y$ are similar or else they are of opposite type and have inverse frequencies, as wanted.
\end{proof}

\begin{corollary}
If $X$ and $Y$ are skew Sturmian sequences then they are flow equivalent.
\end{corollary}
\begin{proof}
A direct consequence of Theorem~\ref{floweq} and Proposition~\ref{anomalysize}.
\end{proof}

\end{document}